\documentclass{amsart}
\usepackage{amsmath,amssymb,bm}
\usepackage[dvips]{graphicx}
\usepackage{color}
\theoremstyle{definition}
\newtheorem{defn}{Definition}[section]
\newtheorem{example}[defn]{Example}
\newtheorem{rem}[defn]{Remark}
\theoremstyle{plain}
\newtheorem{thm}[defn]{Theorem}

\newtheorem{lem}[defn]{Lemma}
\newtheorem{cor}[defn]{Corollary}

\newtheorem{question}[defn]{Question}
\title[]{Flat plumbing basket, self-linking number and Thurston-Bennequin number}
\author{Keiji Tagami}
\address{
Department of Mathematics, Faculty of Science and Technology, Tokyo University
of Science, Noda, Chiba, 278-8510, Japan
}
\email{tagami\_keiji@ma.noda.tus.ac.jp}
\date{\today}
\begin{document}
\begin{abstract}
A flat plumbing basket is a surface consisting a disk and finitely many bands which are contained in distinct pages of the trivial open book decomposition of $\mathbf{S}^{3}$. 
In this paper, we construct a Legendrian link from a flat plumbing basket, and we describe a relation among the self-linking number, the Thurston-Bennequin number and the flat plumbing basket number of the link. 
As a corollary, we determine the flat plumbing basket numbers of torus links. 
\end{abstract}
\maketitle
\noindent
{\bf Comments: }
This is a previous version of \cite{Ito-Tagami}. 
In \cite{Ito-Tagami}, we correct and generalize the construction of the Legendrian link $\mathcal{L}_{F}$ associated with a flat plumbing basket $F$ given in Section~\ref{sec:legendrian-fpb}. 
Moreover, we expand and improve the results in this paper. 
We also give a front projection of $\mathcal{L}_F$, which answers Question~\ref{ques:front}. 
\section{Introduction}
It is well known that for any oriented link $L$ in the $3$-sphere $\mathbf{S}^3$, there is some oriented compact surface whose boundary is the link $L$. 
Such a surface is called a Seifert surface of $L$. 
Furihata-Hirasawa-Kobayashi \cite{FHK} introduced a concept of positions of a Seifert surface, which is called  a ``flat plumbing basket". 
\par
A Seifert surface is a {\it flat plumbing basket} if it is obtained from a disk by plumbing some unknotted and untwisted annuli so that the gluing regions are in the disk. 
Flat plumbing baskets can be expressed in terms of the trivial open book decomposition $\mathcal{O}$ of $\mathbf{S}^3$. 
Namely, a flat plumbing basket consists a page $D_{0}$ of $\mathcal{O}$ and finitely many bands which are contained in distinct pages of $\mathcal{O}$. 
\par
We say that an oriented link $L$ admits a flat plumbing basket presentation $F$ if $F$ is a flat plumbing basket and its boundary is $L$. 
\begin{thm}[\cite{FHK}]
Any oriented link admits a flat plumbing basket presentation.
\end{thm}
On flat plumbing baskets, there are some related works (for example, see \cite{CCK, FHK, Hirose-Nakashima, Imoto, Kim-flat, KKL, imoto-thesis}). 
In particular, in \cite{Kim-flat}, the concept of the flat plumbing basket number of a link is introduced. 
The {\it flat plumbing basket number} $fpbk(L)$ of an oriented link $L$ is the minimal number of bands to obtain a flat plumbing basket presentation of the link. 
Hirose-Nakashima \cite{Hirose-Nakashima} gave a lower bound for $fpbk(K)$ of a knot $K$ as follows. 
\begin{thm}[{\cite[Theorem~1.3]{Hirose-Nakashima}}]\label{thm:Hirose-Nakashima}
Let $K$ be a non-trivial knot, and $g(K)$ be the minimal genus of the Seifert surface (i.e. three genus) of $K$. 
For the Alexander polynomial $\Delta_{K}(t)$ of $K$, let $a$ be the coefficient of the term of highest degree, and 
\[
\deg\Delta_{K}(t):=(\text{the highest degree of } \Delta_{K}(t))-(\text{the lowest degree of }\Delta_{K}(t)). 
\]
Then $fpbk(K)$ is evaluated as follows.
\begin{enumerate}
\item If $a=\pm 1$, then $fpbk(K)\geq 2g(K)+2$. 
\item If $a\neq\pm 1$, then $fpbk(K)\geq \max\{2g(K)+2, \deg \Delta_{K}(t)+4\}$. 
\end{enumerate}
\end{thm}
\par 
As mentioned above, flat plumbing baskets can be expressed in terms of an open book decomposition. 
On the other hand, by Thurston-Winkelnkemper's work \cite{Thurston-Winkelnkemper}, we can construct a contact structure of $\mathbf{S}^3$ from an open book decomposition. 
Hence, it seems that there are some relations between flat plumbing baskets and contact topology. 
\par
In this paper, we construct a Legendrian link $\mathcal{L}_{F}$ from a flat plumbing basket $F$ (see Section~\ref{sec:legendrian-fpb}). 
By using this Legendrian link, we give two inequalities on flat plumbing basket numbers as follows. 
\begin{thm}\label{thm:self-linking}
Let $L$ be a non-trivial oriented link in $\mathbf{S}^3$. 
Then we have 
\[
\max\{-\overline{sl}(L), -\overline{sl}(\overline{L})\}-1\leq fpbk(L), 
\]
where $\overline{L}$ is the mirror image of $L$ and $\overline{sl}(L)$ is the maximal self-linking number of $L$ (for detail see Section~$\ref{sec:self-linking}$ and Theorem~$\ref{thm:self-linking2}$).  
\end{thm}
\begin{thm}\label{thm:main}
Let $L$ be a non-trivial oriented link in $\mathbf{S}^3$. 
Suppose that $L$ has no split component which is isotopic to the unknot. 
Then, we obtain 
\[
\max\{-\overline{tb}(L), -\overline{tb}(\overline{L})\}+2\leq 2fpbk(L), 
\]
where $\overline{tb}(L)$ is the maximal Thurston-Bennequin number of $L$. 
The equality holds if and only if $L$ is a non-trivial alternating torus link $T_{2,n}$ for some $|n|\geq 2$. 
\end{thm}
As a corollary, we determine $fpbk(T_{p,q})$ for the $(p,q)$-torus link $T_{p,q}$ (see Corollary~\ref{cor:torus}). 
Moreover, by using Theorems~\ref{thm:self-linking} and \ref{thm:main}, we improve \cite[Table~1]{Hirose-Nakashima} (see Table~\ref{table1}). 
\par
This paper is organized as follows: 
In Section~\ref{sec:preliminary}, we recall the definitions of open book decompositions, flat plumbing baskets and contact structures. 
In Section~\ref{sec:legendrian-fpb}, we construct a Legendrian link from a flat plumbing basket. 
Moreover, we describe a relation among the maximal self-linking numbers, the maximal Thurston-Bennequin numbers and the flat plumbing basket numbers. 
In Section~\ref{sec:self-linking}, we prove Theorem~\ref{thm:self-linking} (Theorem~\ref{thm:self-linking2}). 
In Section~\ref{sec:thurston-bennequin}, we prove Theorem~\ref{thm:main}. 
In Section~\ref{sec:discussion}, we give further observations. 
\section{Preliminary}\label{sec:preliminary}
\subsection{Trivial open book decomposition and flat plumbing basket}
Let $M$ be an oriented closed $3$-manifold. 
Suppose that for a link $L$ in $M$, there is a fiber projection $\pi:M\setminus L\rightarrow \mathbf{S}^1$ such that each fiber is the interior of a Seifert surface of $L$. 
Then, $(L, \pi)$ is called an {\it open book decomposition} of $M$. The closure of each fiber is called a {\it page}, and $L$ is called the {\it binding}.
\par 
Let $U$ be the unknot in $\mathbf{S}^3$. 
The knot complement $\mathbf{S}^3\setminus U$ is a product $\operatorname{Int}(D^2)\times \mathbf{S}^1$. 
Hence, there exists a fiber projection $\pi\colon \mathbf{S}^3\setminus U\rightarrow \mathbf{S}^1$ whose fibers $\pi^{-1}(D_{\theta})$ are open disks for $0\leq \theta <2\pi$. 
Put $D_{\theta}:=\overline{\pi^{-1}(D_{\theta})}$. 
An orientation of $U$ induces an orientation of each fiber $D_{\theta}$ and a positive direction of the fibration $\{D_{\theta}\}_{0\leq \theta<2\pi}$ (see Figure~\ref{fig:openbook}). 
This fibration is called the {\it trivial open book decomposition} of $\mathbf{S}^3$, denoted by $\mathcal{O}$. 
Then, the unknot is the binding, and each $D_{\theta}$ is a page of $\mathcal{O}$. 
Throughout this paper, we consider the trivial open book decomposition. 
\begin{figure}[h]
\begin{center}
\includegraphics[scale=0.32]{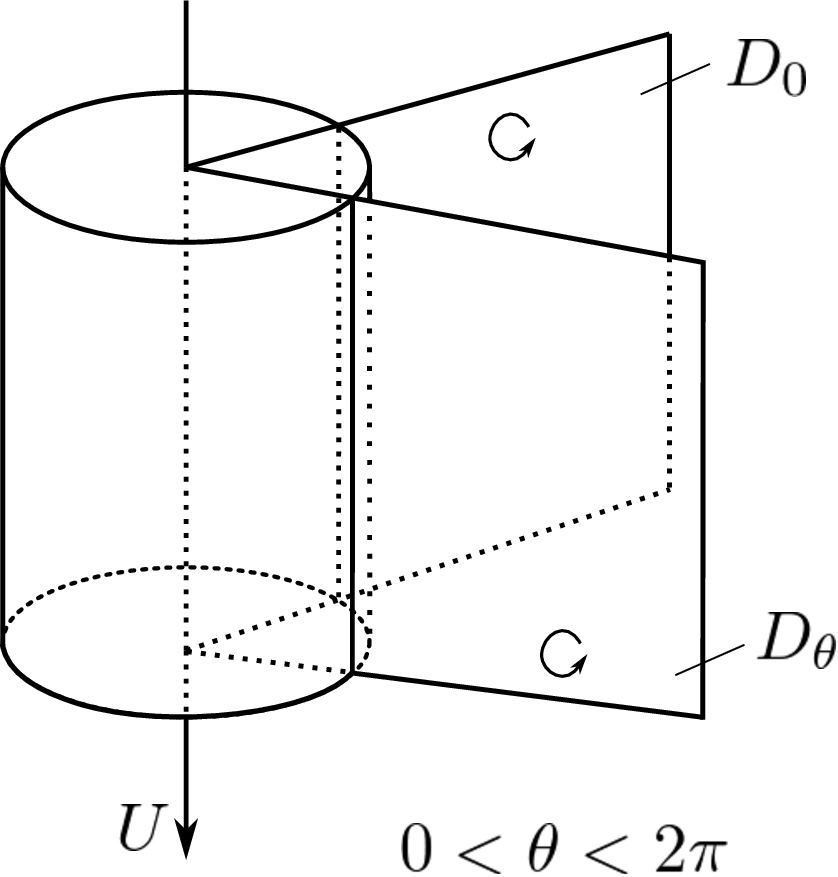}
\end{center}
\caption{Open book decomposition}
\label{fig:openbook}
\end{figure}
\par
A Seifert surface $F$ is a {\it flat plumbing basket} if there are finitely many bands $B_{1}, \dots, B_{n}$ and $0<\theta_{1}<\dots<\theta_{n}<2\pi$ such that 
$F=D_{0}\cup B_{1} \cup \dots \cup B_{n}$, each band $B_{i}$ is contained in $D_{\theta_{i}}$ and $B_{i}\cap U$ consists of two arcs. 
We call the subscript $i$ of $B_{i}$ the {\it label} of the band. 
A flat plumbing basket $F$ is a {\it flat plumbing basket presentation} of an oriented link $L$ if $\partial F$ is ambient isotopic to $L$ (for example, see Figure~\ref{fig:fpb1}). 
\par 
For a flat plumbing basket $F$, by recording the labels of the bands as one travel along $U=\partial D_{0}$, one obtain a cyclic word $W_{F}$ in $\{1, \dots, n\}$ such that each letter appears exactly twice, where $n$ is the number of the bands of $F$. 
We call $W_{F}$ the {\it flat basket code for $F$}. 
\par 
We define the {\it flat plumbing basket number} $fpbk(L)$ of an oriented link $L$ to be the minimal number of bands to obtain a flat plumbing basket presentation of $L$ from $D_{0}$. 
Namely, 
\[
fpbk(L):=\min\{b_{1}(F)\mid F\text{ is a flat plumbing basket presentation of }L\}, 
\]
where $b_{1}(F)$ is the first betti number of $L$. 
We remark that $fpbk(L)\in 2\mathbf{Z}+|L|-1$, where $|L|$ is the number of the components of $L$, and $fpbk$ is preserved under taking mirror image. 
\begin{figure}[h]
\begin{center}
\includegraphics[scale=0.55]{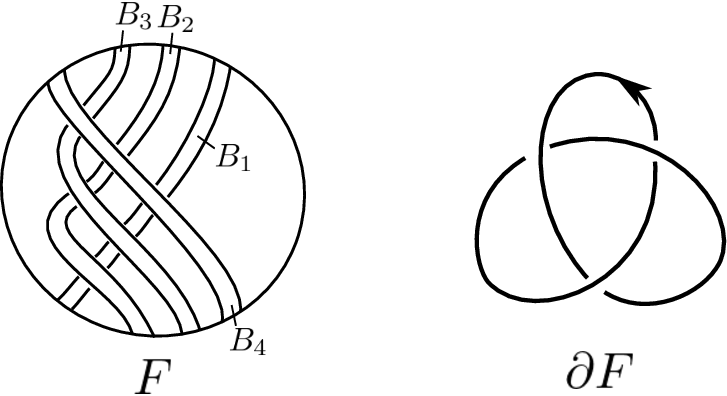}
\end{center}
\caption{An example of a flat plumibing basket $F$. It is a flat plumbing presentation of the negative trefoil. The flat basket code is $W_{F}=(1,2,3,4,1,2,3,4)$.}
\label{fig:fpb1}
\end{figure}
\subsection{Trivial open book decomposition and standard contact structure}
In this section, we recall the relation between open book decompositions and contact structures. 
For detail, for example, see \cite{OS}. 
\par
Let $M$ be an oriented closed $3$-manifold. 
A $1$-form $\alpha \in \Omega^{1}(M)$ is a {\it contact form} if $\alpha \wedge d\alpha$ is nowhere $0$. 
A $2$-dimensional distribution $\xi \subset TM$ is a {\it contact structure} if there is a contact form $\alpha$ such that $\xi=\operatorname{Ker}\alpha$. 
A contact structure $\xi$ on $M$ is {\it supported} by an open book decomposition $(L, \pi)$ if $\xi$ can be represented by a contact form $\alpha$ such that the binding is a transverse to $\xi$, $d\alpha$ is a volume form on every page and the orientation of binding induced by $\alpha$ agrees with the orientation of the pages. 
\par  
It is known that for any open book decomposition of $M$, we can construct a contact structure on $M$ supported by the open book decomposition, by Thurston-Winkelnkemper's construction \cite{Thurston-Winkelnkemper}. 
By Giroux's work \cite{Giroux1}, such a contact structure is unique up to contact isotopy.  
In particular, the trivial open book decomposition $\mathcal{O}$ of $\mathbf{S}^3$ supports the standard contact structure $\xi_{std}$ of $\mathbf{S}^3$. 
See Figure~\ref{fig:standard}. 
\par
For a surface $\Sigma$ in $\mathbf{S}^3$, we consider $\xi_{std}\cap T\Sigma $. 
For generic $\Sigma$, this intersection is a line field except at finitely many points where $\Sigma$ is tangent to $\xi_{std}$. 
By integrating $\xi_{std}\cap T\Sigma$, we obtain a foliation of $\Sigma$ with singularities. 
Such a foliation is called the {\it characteristic foliation} of $\Sigma$ in $\xi_{std}$. 
In the next section, Section~\ref{sec:legendrian-fpb}, we will construct a Legendrian link $\mathcal{L}_{F}$ from a flat plumbing basket $F$, which is isotopic to the boundary $\partial F$. 
Then, it is convenient to see the characteristic foliation of the boundary of a tubular neighborhood of $U$ and each fiber $D_{\theta}$ (see Figures~\ref{fig:foliation} and \ref{fig:foliation2}). 
\begin{figure}[h]
\begin{center}
\includegraphics[scale=0.5]{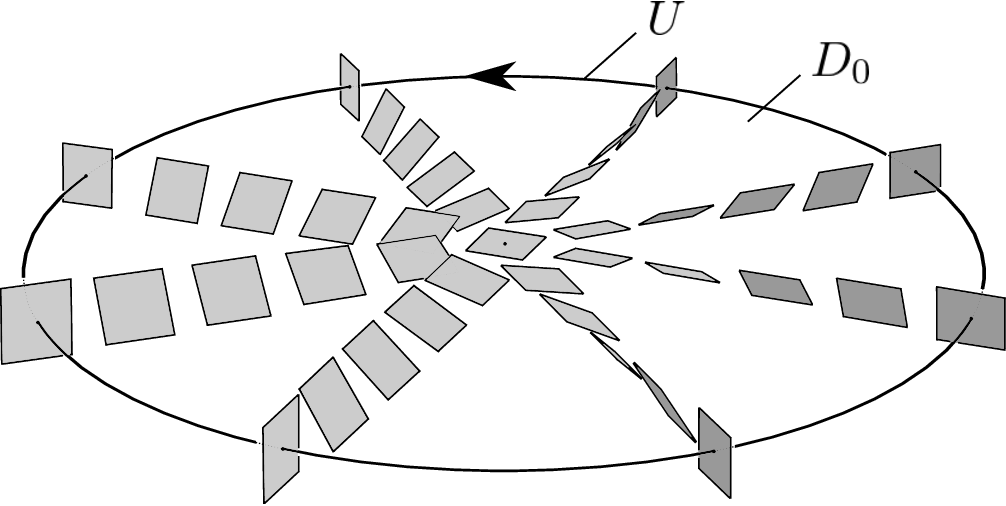}
\end{center}
\caption{The standard contact structure $\xi_{std}$ on $\mathbf{S}^3$ }
\label{fig:standard}
\end{figure}
\begin{figure}[h]
\begin{center}
\includegraphics[scale=0.45]{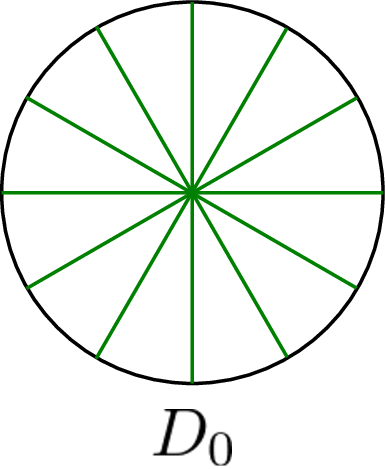}
\end{center}
\caption{Characteristic foliation of $D_{0}$. This characteristic foliation has one singular point. The characteristic foliations of other fibers $D_{\theta}$ are similar to that of $D_{0}$. In particular there is one singular point. }
\label{fig:foliation}
\end{figure}
\begin{figure}[h]
\begin{center}
\includegraphics[scale=0.45]{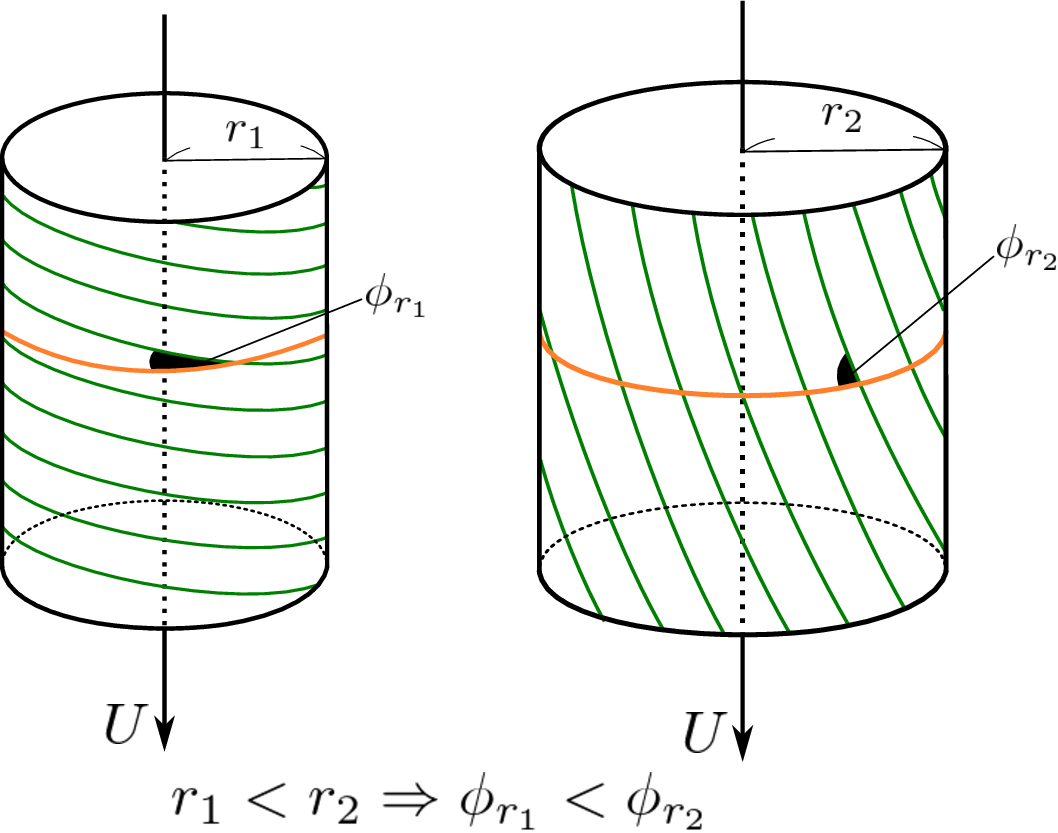}
\end{center}
\caption{Characteristic foliation of the boundary of a tubular neighborhood of the binding $U$. If we increase the radius of the tube, then the angle $\phi_{r}$ between the foliation and the meridian also increases. }
\label{fig:foliation2}
\end{figure}
%
%
\section{Legendrian link from flat plumbing basket}\label{sec:legendrian-fpb}
%
Let $\mathcal{L}$ be an oriented link in $\mathbf{S}^3$. 
Then, $\mathcal{L}$ is a {\it Legendrian link} in $\xi_{std}$ if it is tangent to $\xi_{std}$. 
In this section, we construct a Legendrian link from a flat plumbing basket. 
\par 
Let $F$ be a flat plumbing basket with $b_{1}(F)>0$. 
Let $N(U)$ be a tubular neighborhood of the binding $U$. 
Since $U$ is transverse to $\xi_{std}$, the boundary $\partial F$ is not a Legendrian link. 
However, since each page is ``almost" tangent to $\xi_{std}$, we can regard $\partial F\setminus N(U)$ as a disjoint union of Legendrian arcs. 
More precisely, we can approximate each component of $\partial F\setminus N(U)$ by a Legendrian arc without changing the link type of $\partial F$. 
To see this, we thin each band $B_{i}$ sufficiently, and we move the band without moving $B_{i}\cap U$ so that the core of $B_{i}$ is on two leaves of the characteristic foliation of $D_{\theta_{i}}$. 
See Figure~\ref{fig:band-move}. 
Then, there is a small perturbation (isotopy) $f_{t}\colon \mathbf{S}^3\rightarrow \mathbf{S}^3$ such that $f_{0}=\operatorname{id}$ and $f_1(\partial B_{i}\setminus N(U))$ is Legendrian (by Legendrian realization principle \cite{Honda}). 
After this operation, $\partial B_{i}\setminus N(U)$ can be regarded as Legendrian arcs. 
\par
As mentioned above, we see that $\partial F\cap N(U)$ is not Legendrian. 
In order to construct a Legendrian link from $\partial F$, we replace each component $\beta$ of $\partial F\cap N(U)$ by an arc $\alpha$ depicted in Figure~\ref{fig:replacement}. 
Each $\alpha$ is in a leaf of the characteristic foliation of $\partial (N(U))$, connects the two points of $\partial \beta$ and satisfies $|\alpha\cap D_{0}|=1$. 
We can take such an $\alpha$ by adjusting the radius of $N(U)$ locally.  
\par
By this replacement (and smoothening the curve at the points of $\partial F\cap \partial(N(U))$ in $\xi_{std}$), we obtain a Legendrian link from $\partial F$. 
We denote the Legendrian link by $\mathcal{L}_{F}$, and called the {\it Legendrian link associated with $F$}.  
For example, see Figure~\ref{fig:example}. 
\par
The classical invariants, the Thurston Bennequin number $tb$ and the rotation number $rot$, of $\mathcal{L}_{F}$ can be computed as follows. 
\begin{figure}[h]
\begin{center}
\includegraphics[scale=0.5]{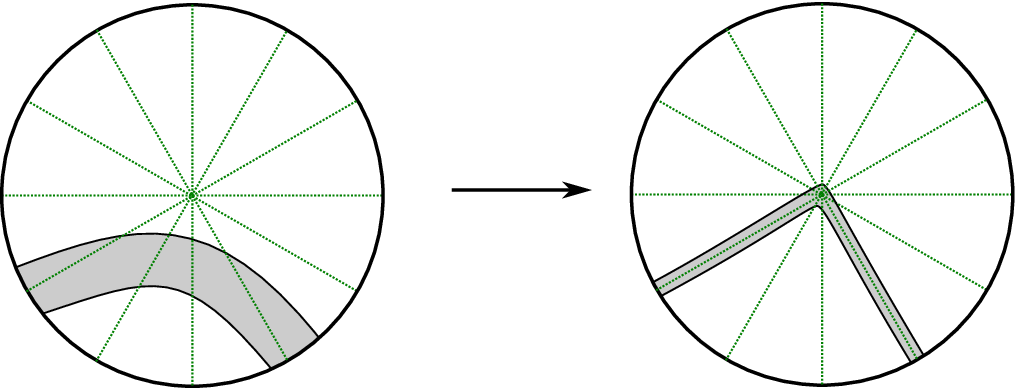}
\end{center}
\caption{Thinning and moving a band $B_{i}$ so that the core of $B_{i}$ is on two leaves of the characteristic foliation of $D_{\theta_{i}}$. }
\label{fig:band-move}
\end{figure}
\begin{figure}[h]
\begin{center}
\includegraphics[scale=0.5]{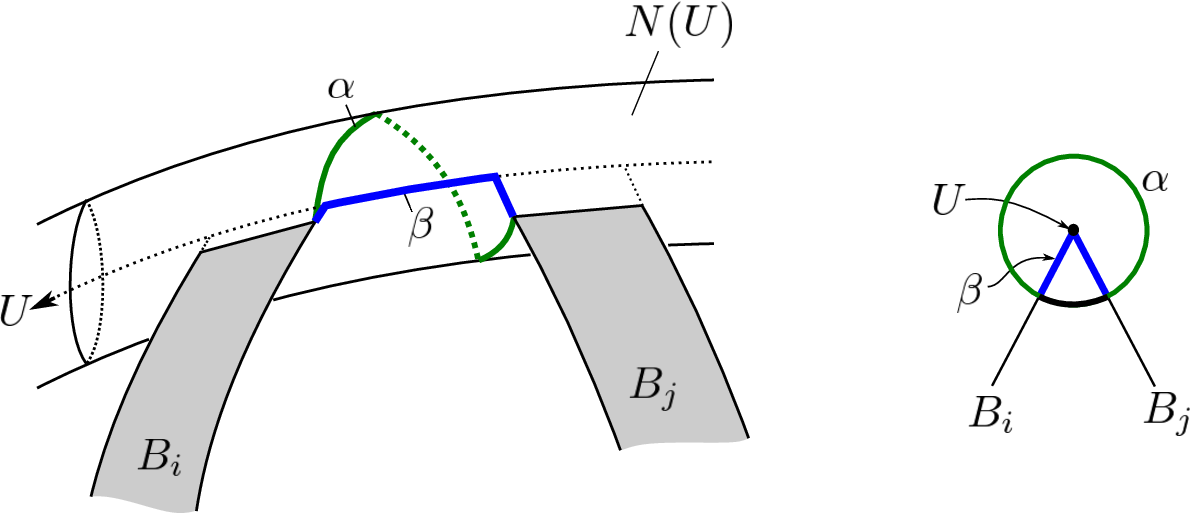}
\end{center}
\caption{
In order to construct a Legendrian link $\mathcal{L}_{F}$ from $\partial F$, we replace each component $\beta$ of $\partial F\cap N(U)$ by an arc $\alpha$. 
The arc $\alpha$ is contained in a leaf of the characteristic foliation of $\partial (N(U))$ (see Figure~\ref{fig:foliation2}), connects the two points of $\partial \beta$ and satisfies $|\alpha\cap D_{0}|=1$. 
By adjusting the radius of $N(U)$ locally, we can take such an $\alpha$. 
}
\label{fig:replacement}
\end{figure}
\begin{figure}[h]
\begin{center}
\includegraphics[scale=0.6]{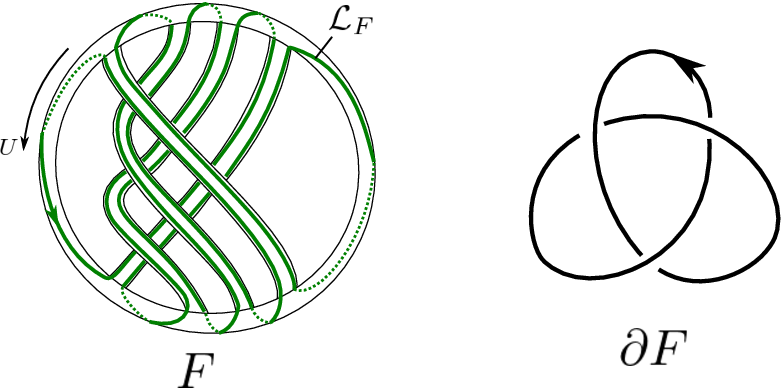}
\end{center}
\caption{An example of $\mathcal{L}_{F}$. This is isotopic to the negative trefoil knot. Precisely, the radius of $N(U)$ is non-uniform. }
\label{fig:example}
\end{figure}
\begin{lem}\label{lem:main}
Let $F$ be a flat plumbing basket with $b_1(F)>0$ and $\mathcal{L}_{F}$ be the Legendrian link associated with $F$. 
Then, we obtain 
\[
tb(\mathcal{L}_{F})=-2b_{1}(F), 
\]
where $tb(\mathcal{L}_{F})$ is the Thurston-Bennequin number of $\mathcal{L}_{F}$. 
\end{lem}
\begin{proof}
Let $\mathcal{L}_{F}^{+}$ be a Legendrian link obtained by pushing of $\mathcal{L}_{F}$ in the direction of a nonzero vector field transverse to $\xi_{std}$. 
Then the Thurston-Bennequin number $tb(\mathcal{L}_{F}) $ is computed by 
\[
tb(\mathcal{L}_{F})=lk(\mathcal{L}_{F}, \mathcal{L}_{F}^{+}),  
\]
where $lk(\mathcal{L}_{F}, \mathcal{L}_{F}^{+})$ is the linking number between $\mathcal{L}_{F}$ and $\mathcal{L}_{F}^{+}$. 
By Figure~\ref{fig:contribution}, we see that each band of $F$ contributes $-2$ to the linking number. 
Note that the crossings between distinct bands do not contribute the linking number. 
Hence, we finish the proof. 
\end{proof}
\begin{figure}[h]
\begin{center}
\includegraphics[scale=0.45]{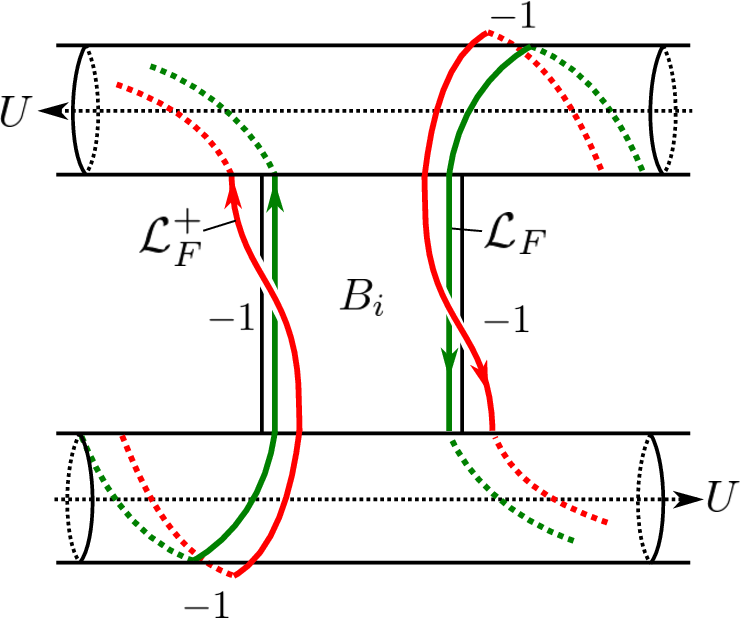}
\end{center}
\caption{Near a band of a flat plumbing basket. Each band contributes $-4/2=-2$ to the linking number $lk(\mathcal{L}_{F}, \mathcal{L}_{F}^{+})$. }
\label{fig:contribution}
\end{figure}
\begin{lem}\label{lem:rotation}
Let $F$ be a flat plumbing basket with $b_1(F)>0$ and $\mathcal{L}_{F}$ be the Legendrian link associated with $F$. 
Suppose that $\mathcal{L}_{F}$ has the orientation which agrees with the orientation of the binding $U$. 
Then, we obtain 
\[
rot(\mathcal{L}_{F})=-b_{1}(F)+1, 
\]
where $rot(\mathcal{L}_{F})$ is the rotation number of $\mathcal{L}_{F}$. 
\end{lem}
\begin{proof}
The rotation number $rot (\mathcal{L})$ of an oriented Legendrian link $\mathcal{L}$ is the winding number of $T\mathcal{L}$ with respect to a trivialization of $\xi_{std}$ along $\mathcal{L}$. 
Let $W_{F}=(i_{1}, i_{2}, \dots, i_{2n})$ be the flat basket code of $F$, where $n=b_{1}(F)$. 
For convenience, define $i_{2n+1}=i_1$. 
Let $\alpha$ be the arc in $\mathcal{L}_{F}$ which connects $B_{i_{k}}$ and $B_{i_{k+1}}$ as in Figure~\ref{fig:rotation1}. 
Define $\zeta_{k}$ be the angle corresponding to the arc in $U$ which connects $B_{i_{k}}$ and $B_{i_{k+1}}$. 
Then, as in Figure~\ref{fig:rotation1}, the arc $\alpha$ contributes 
\[
-2\pi+(\theta_{i_{k+1}}-\theta_{i_{k}})+\frac{\pi}{2}\times 2+\zeta_{k}
\]
to $rot (\mathcal{L}_{F})$. 
Moreover, when we go across a band $B_{i}$ twice along $\mathcal{L}_{F}$, the winding number increases $\xi_{i}$, where $\xi_{i}$ is the angle corresponding to the two arcs in $B_{i}\cap U$ (see Figure~\ref{fig:rotation2}). 
Then, we obtain 
\begin{align*}
2\pi \times rot (\mathcal{L}_{F})
&=\sum_{k=1}^{2n}(-2\pi+(\theta_{i_{k+1}}-\theta_{i_{k}})+\frac{\pi}{2}\times 2+\zeta_{k})+\sum_{i=1}^{n}\xi_{i}\\
&=-2n\pi+\sum_{k=1}^{2n}\zeta_{k}+\sum_{i=1}^{n}\xi_{i}\\
&=-2n\pi+2\pi. 
\end{align*}
Hence, we have $rot (\mathcal{L}_{F})=-n+1=-b_{1}(F)+1$. 
\begin{figure}[h]
\begin{center}
\includegraphics[scale=0.5]{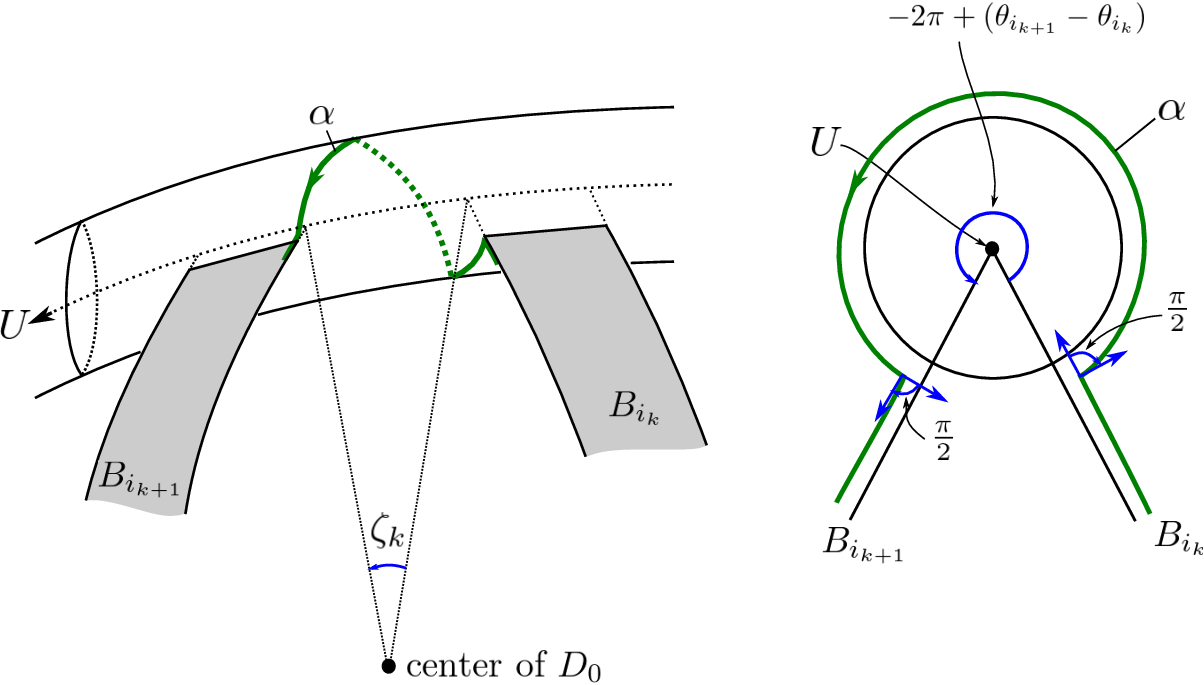}
\end{center}
\caption{}
\label{fig:rotation1}
\end{figure}
\begin{figure}[h]
\begin{center}
\includegraphics[scale=0.45]{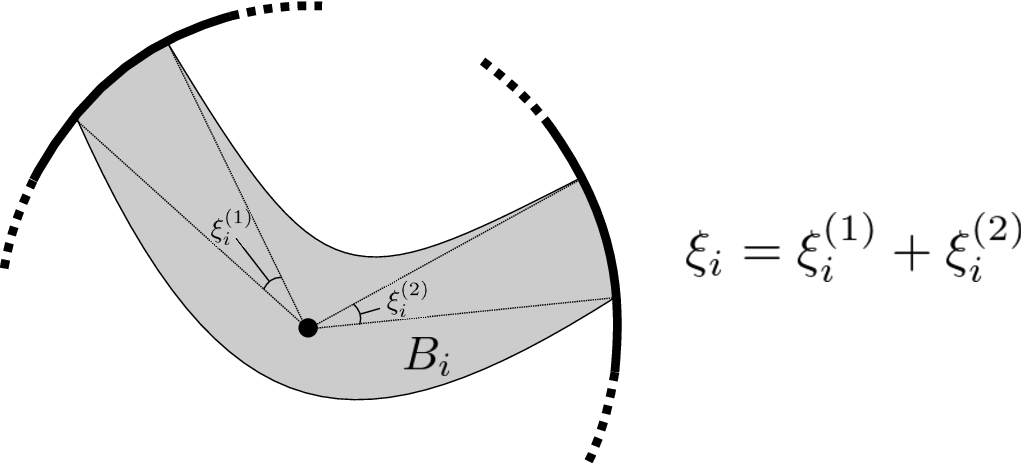}
\end{center}
\caption{}
\label{fig:rotation2}
\end{figure}
\end{proof}
%
\section{Maximal self-linking number and flat plumbing basket}\label{sec:self-linking}
Let $l$ be an oriented link in $\mathbf{S}^3$. 
Then, $l$ is a {\it transverse link} in $\xi_{std}$ if it is positively transverse to $\xi_{std}$. 
The {\it self-linking number} $sl(l)$ of $l$ is defined as the linking number $lk(l, l')$ of $l$ and $l'$, where $l'$ is a push-off of $l$ obtained by a non-zero vector field in $\xi_{std}$. 
It is known that for any Legendrian link $\mathcal{L}$, by pushing $\mathcal{L}$ in a sufficiently small annulus neighborhood, we can construct two transverse link $l_+$ and $l_-$ such that $l_{\pm}$ is isotopic to $\mathcal{L}$ topologically and 
\[
sl(l_{\pm})=tb(\mathcal{L})\mp rot(\mathcal{L}). 
\]
For example, see \cite{Geiges, OS}. 
\par
Let $F$ be a flat plumbing basket with $b_{1}(F)>0$ and $\partial F=L$. 
Let $\mathcal{L}_{F}$ be the Legendrian link associated with $F$. 
By Lemmas~\ref{lem:main} and \ref{lem:rotation}, we can construct a transverse link $l_{F}$ such that $l_{F}$ is isotopic to $L$ and 
\[
sl(l_F)=tb(\mathcal{L}_F)+ |rot(\mathcal{L}_F)|=-b_1(F)-1\in 2\mathbf{Z}+|L|. 
\]
Hence, we obtain the following. 
\begin{thm}[Theorem~\ref{thm:self-linking}]\label{thm:self-linking2}
Let $L$ be a non-trivial oriented link in $\mathbf{S}^3$. 
Define the maximal self-linking number $sl(L)$ of $L$ as 
\[
\overline{sl}(L):=\max\{tb(\mathcal{L})+|rot(\mathcal{L})|\mid \mathcal{L}\text{ is a Legendrian link in $\xi_{std}$ and isotopic to }L\}. 
\]
Then we have 
\[
\max\{-\overline{sl}(L), -\overline{sl}(\overline{L})\}-1\leq fpbk(L)=fpbk(\overline{L}), 
\]
where $\overline{L}$ is the mirror image of $L$.  
\end{thm}
\begin{cor}\label{cor:torus}
For any $p\geq q>1$, we have 
\[
pq-p+q-1=-\overline{sl}(T_{p,-q})-1= fpbk(T_{p,q}), 
\]
where $T_{p,q}$ is the positive $(p,q)$-torus link and 
$T_{p,-q}=\overline{T_{p,q}}$. 
\end{cor}
\begin{proof}
By Morton-Franks-Williams (MFW) inequality \cite{F-W, Morton}, we have 
\[
-2b(L)\leq \overline{sl}(L)+\overline{sl}(\overline{L})\leq -\operatorname{breadth}_{v}P_{L}(v,z)-2, 
\]
where $b(L)$ is the braid index of $L$ and $P_{L}(v,z)$ is the HOMFLYPT polynomial. 
Franks and Williams \cite{F-W} proved that for any torus link, MFW inequality is sharp, that is $2q=\operatorname{breadth}_{v}P_{T_{p,q}}(v,z)+2$. 
Moreover, it is known that $\overline{sl}(T_{p,q})=pq-p-q$ (\cite{Bennequin}). 
Hence, we have 
\[
-\overline{sl}(T_{p,-q})\geq \overline{sl}(T_{p,q})+\operatorname{breadth}_{v}P_{T_{p,q}}(v,z)+2=pq-p+q. 
\]
On the other hand, by \cite[Theorem~2.4]{FHK}, we see that $T_{p,q}$ has a flat plumbing basket presentation $F$ with $b_1(F)=pq-p+q-1$. 
By Theorem~\ref{thm:self-linking2}, we have 
\[
-\overline{sl}(T_{p,-q})\leq fpbk(T_{p,q})+1\leq pq-p+q, 
\]
and we finish the proof. 
\end{proof}
\begin{cor}
Let $L$ be an oriented link with $\overline{sl}(L)=-\chi(L)$, where $\chi(L)$ is the maximal Euler characteristic of $L$ (for example, if $L$ is strongly quasipositive, $L$ satisfies this condition). 
Then we have 
\[
1-\chi(L)+\operatorname{breadth}_{v}P_{L}(v,z)\leq fpbk(L).
\]
\end{cor}
\begin{proof}
By MFW inequality and Theorem~\ref{thm:self-linking2}, we have 
\[
\overline{sl}(L)+\operatorname{breadth}_{v}P_{L}(v,z)+2 \leq -\overline{sl}(\overline{L})\leq fpbk(L)+1. 
\]
By the assumption, we finish the proof. 
\end{proof}
\begin{cor}\label{cor:self-linking-homfly}
Let $L$ be an oriented link. Then, we have 
\[
\operatorname{maxdeg}_{v}P_{L}(v,z)\leq -\overline{sl}(L)-1\leq fpbk(L). 
\]
\end{cor}
\begin{proof}
The first inequality is the HOMFLYPT bound on the self-linking number, which follows from MFW inequality. 
The second follows from Theorem~\ref{thm:self-linking2}. 
\end{proof}
\begin{cor}\label{cor:twist-knot}
Let $K_{m}$ be the $m$-twist knot (Figure~\ref{fig:twist-knot}). 
Then, for any $k\geq 0$, we have 
$2k\leq fpbk(K_{2k})$ and $2k+4\leq fpbk(K_{2k+1})$. 
\end{cor}
\begin{proof}
By \cite[Theorem~1.2]{ENV}, we see that 
\begin{itemize}
\item $\max\{-\overline{sl}(K_{2k}), -\overline{sl}(\overline{K_{2k}})\}\geq 2k+1$, 
\item $\max\{-\overline{sl}(K_{2k+1}), -\overline{sl}(\overline{K_{2k+1}})\}\geq 2k+5$. 
\end{itemize}
By Theorem~\ref{thm:self-linking2}, we finish the proof. 
\end{proof}
\begin{rem}
Mikami Hirasawa showed that $fpbk(K_{2k+1})\leq 2k+4$ for $k\geq 0$, and $fpbk(K_{2k})\leq 2k$ for $k\geq 3$ in his forthcoming paper. 
Hence, by Corollary~\ref{cor:twist-knot}, we have 
\begin{itemize}
\item $ fpbk(K_{2k+1})=2k+4$ for $k\geq 0$, 
\item $fpbk(K_{2k})=2k$ for $k\geq 3$, 
\item $fpbk(K_{2})=4$ and $fpbk(K_{4})=6$. 
\end{itemize}
\end{rem}
\begin{figure}[h]
\begin{center}
\includegraphics[scale=0.8]{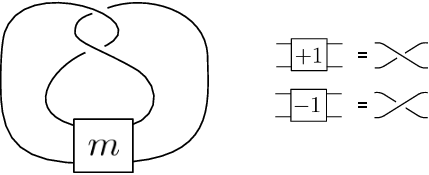}
\end{center}
\caption{The $m$-twist knot $K_{m}$ }
\label{fig:twist-knot}
\end{figure}
\begin{example}\label{ex:trefoil1}
By Corollary~\ref{cor:torus}, we have 
\[
fpbk(3_1)=4=\operatorname{maxdeg}_{v}P_{3_1}(v,z). 
\]
Moreover, since the right-hand side is additive and $fpbk$ is subadditive under connected sum, by Corollary~\ref{cor:self-linking-homfly}, we have 
\[
fpbk(\sharp_{n}3_1)=4n. 
\]
\end{example}
%
\section{Maximal Thurston-Bennequin number and flat plumbing basket}\label{sec:thurston-bennequin}
Let $L$ be an oriented link in $\mathbf{S}^3$. The {\it maximal Thurston-Bennequin number} $\overline{tb}(L)$ of $L$ is the maximal number of Thurston-Bennequin numbers of Legendrian links in $\xi_{std}$ which are isotopic to $L$. 
In this section, we compare $tb(\mathcal{L}_{F})$ with $\overline{tb}(L)$, where $L=\partial F$. 
\par
Let $\alpha$ be an arc used in the construction of $\mathcal{L}_{F}$. 
Suppose that $\alpha$ goes round the meridian of $\partial (N(U))$. 
See the left picture of Figure~\ref{fig:shortcut}. 
If we increase the radius of $N(U)$, the angle between the characteristic foliation and the meridian also increases as in Figure~\ref{fig:foliation2}. 
Hence, by increasing the radius of $N(U)$ locally, we can obtain a new Legendrian arc $\alpha'$ instead of $\alpha$ (see Figure~\ref{fig:shortcut}). 
Then, we can construct a new Legendrian link by replacing $\alpha$ in $\mathcal{L}_{F}$ with $\alpha'$ as in Figure~\ref{fig:shortcut}. 
We call this operation a {\it shortcut} for $\mathcal{L}_{F}$. 
By considering the Lagrangian projection, we see that a shortcut corresponds to a (de)stabilization in a Lagrangian projection of $\mathcal{L}_{F}$. 
\begin{lem}\label{lem:shortcut1}
Let $F$ be a flat plumbing basket. 
Let $\mathcal{L}_{F}'$ be a Legendrian link obtained from $\mathcal{L}_{F}$ by taking one shortcut. 
Then, we obtain $tb(\mathcal{L}_{F}')=tb(\mathcal{L}_{F})+1$. 
\end{lem}
\begin{proof}
See Figure~\ref{fig:shortcut2}. 
\end{proof}
\begin{figure}[h]
\begin{center}
\includegraphics[scale=0.7]{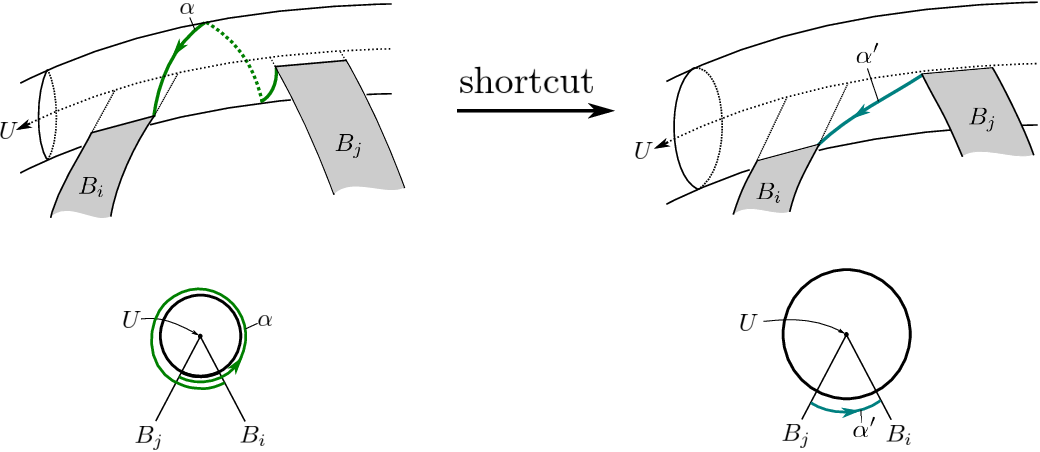}
\end{center}
\caption{Shortcut operation. In this picture, $j>i$, that is, $\theta_{j}>\theta_{i}$ (cf. Figure~\ref{fig:replacement}). In this case, we can take a shortcut. Namely we can replace $\alpha$ with $\alpha'$.}
\label{fig:shortcut}
\end{figure}
\begin{figure}[h]
\begin{center}
\includegraphics[scale=0.79]{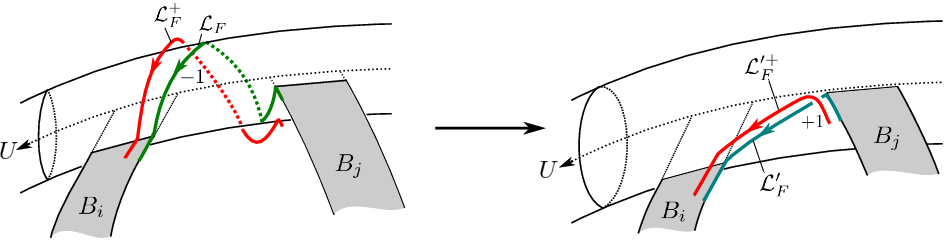}
\end{center}
\caption{One shortcut contributes $(+1-(-1))/2=+1$ to the Thurston-Bennequin number. }
\label{fig:shortcut2}
\end{figure}
\begin{lem}\label{lem:shortcut2}
Let $F$ be a flat plumbing basket with $b_{1}(F)=n\geq 2$. 
Suppose that $\partial F$ has no split component which is isotopic to the unknot. 
Then, we can take at least two shortcuts for $\mathcal{L}_{F}$. 
\end{lem}
\begin{proof}
Let $W_{F}$ be the flat basket code of $F$. 
Because $\partial F$ has no split component which is isotopic to the unknot, $W_{F}$ has no subcode $(n, n)$. 
Since the cyclic word $W_{F}$ has exactly two $n$, there are two subcodes $(n, i_{1})$ and $(n, i_{2})$ of $W_{F}$ for some $i_{1}, i_{2}\in \{1,\dots, n-1\}$. 
At the corresponding places in $F$ to subcodes $(n, i_{1})$ and $(n, i_{2})$, we can take shortcuts for $\mathcal{L}_{F}$ since $n>i_1$ and $n>i_2$. 
\end{proof}
By Lemmas~\ref{lem:main}, \ref{lem:shortcut1} and \ref{lem:shortcut2}, we obtain the following. 
\begin{lem}\label{lem:main1}
Let $L$ be a non-trivial oriented link. 
Suppose that $L$ has no split component which is isotopic to the unknot. 
Then, we obtain 
\[
-\overline{tb}(L)+2\leq 2fpbk(L). 
\]
\end{lem}
\begin{proof}
Let $F$ be a flat plumbing basket presentation of $L$ with $b_1(F)=fpbk(L)$. 
By Lemma~\ref{lem:main}, we have $tb(\mathcal{L}_{F})=-2b_{1}(F)$. 
By Lemmas~\ref{lem:shortcut1} and \ref{lem:shortcut2}, we obtain $tb(\mathcal{L}_{F})\leq \overline{tb}(L)-2$. 
Hence, we have $-\overline{tb}(L)+2\leq 2b_{1}(F)$. This implies 
$
-\overline{tb}(L)+2\leq 2fpbk(L). 
$
\end{proof}
\begin{example}
Let $F$ be the flat plumbing basket depicted in Figure~\ref{fig:fpb1}, which presents the negative trefoil knot. 
The Legendrian link $\mathcal{L}_{F}$ associated with $F$ is as in Figure~\ref{fig:example}. 
Its Thurston-Bennequin number $tb(\mathcal{L}_{F})$ is $-8$. 
We can take $2$ shortcuts for $tb(\mathcal{L}_{F})$. 
Hence, $tb(\mathcal{L}_{F})+2\leq \overline{tb}(3_{1})$. 
It is known that $\overline{tb}(3_{1})=-6$. 
We see that the Legendrian link obtained by taking $2$ shortcuts for $\mathcal{L}_{F}$ attains the maximal Thurston-Bennequin number of $3_{1}$. 
Moreover, we obtain 
\[
8=-\overline{tb}(3_{1})+2\leq 2fpbk(3_{1})\leq 2\times 4=8. 
\]
Hence, we have $fpbk(3_{1})=4$. 
This coincides the result of Example~\ref{ex:trefoil1}. 
\end{example}
On the equality of Lemma~\ref{lem:main1}, we obtain the following. 
\begin{lem}\label{lem:main2}
Let $L$ be a non-trivial oriented link. 
Suppose that $L$ has no split component which is isotopic to the unknot. 
Then, $L$ satisfies  
\[
-\overline{tb}(L)+2=2fpbk(L) 
\]
if and only if $L$ is a negative alternating torus link $\overline{T_{2, n}}$ for some $n\geq 2$. 
\end{lem}
\begin{proof}
Suppose that $L$ satisfies $-\overline{tb}(L)+2=2fpbk(L)$. 
Let $F$ be a flat plumbing basket presentation of $L$ with $b_{1}(F)=fpbk(L)$. 
Then, by Lemma~\ref{lem:main}, we have $tb(\mathcal{L}_{F})=-2b_{1}(F)=-2fpbk(L)=\overline{tb}(L)-2$. 
By Lemmas~\ref{lem:shortcut1} and \ref{lem:shortcut2}, we can take exactly two shortcuts for $\mathcal{L}_{F}$. 
Such a flat plumbing basket is depicted in Figure~\ref{fig:torus} and its boundary is a negative alternating torus link $\overline{T_{2, n}}$ for some $n\geq 2$. 
\par
Conversely, suppose that $L=\overline{T_{2, n}}$ for some $n\geq 2$. 
It is known that $\overline{tb}(\overline{T_{2, n}})=-2n$ (for example, see \cite{Ng}). 
Moreover, by Figure~\ref{fig:torus}, $\overline{T_{2, n}}$ has a flat plumbing basket presentation $F$ with $b_{1}(F)=n+1$. 
Hence, we obtain $2n+2=-\overline{tb}(L)+2=2fpbk(L)\leq 2(n+1)$. This implies the equality. 
\end{proof}
\begin{figure}[h]
\begin{center}
\includegraphics[scale=0.6]{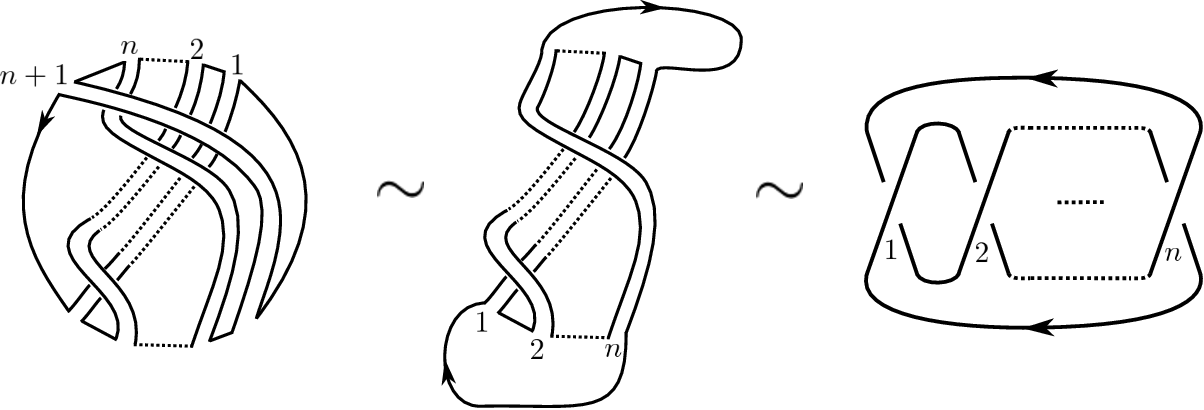}
\end{center}
\caption{A flat plumbing basket presentation with $n+1$ bands. The flat basket code is $(1,2,\cdots, n+1, 1,2,\cdots, n+1)$. We can take exactly two shortcuts for the associated Legendrian link. The boundary is isotopic to the negative torus link $\overline{T_{2,n}}$. }
\label{fig:torus}
\end{figure}
\begin{proof}[Proof of Theorem~\ref{thm:main}]
Note that $fpbk(L)=fpbk(\overline{L})$, where $\overline{L}$ is the mirror image. 
Hence, Theorem~\ref{thm:main} follows from Lemmas~\ref{lem:main1} and ~\ref{lem:main2}. 
\end{proof}
\begin{cor}\label{cor1}
Let $L$ be a non-trivial oriented link. 
Suppose that $L$ has no split component which is isotopic to the unknot. 
If $L$ is not an alternating torus link, we obtain 
\[
\max\{-\overline{tb}(L), -\overline{tb}(\overline{L})\}+3\leq 2fpbk(L).
\]
\end{cor}
\begin{question}
When does the equality of Corollary~\ref{cor1} hold?
Classify such links. 
\end{question}
For example, by Table~\ref{table1}, $8_{21}$ satisfies 
\[
\max\{-\overline{tb}(8_{21}), -\overline{tb}(\overline{8_{21}})\}+3= 2fpbk(8_{21}).
\]
%
\section{Further discussion}\label{sec:discussion}
\subsection{Negativity of links and flat plumbing basket number}
It is known that the maximal self-linking number $\overline{sl}(L)$ and the maximal Thurston-Bennequin number $\overline{tb}(L)$ of a positive link $L$ hold the equality of the Bennequin inequality \cite{Bennequin}, that is, $\overline{sl}(L)=\overline{tb}(L)=-\chi(L)$ (for example, see \cite{T-Tanaka}).  
In this sense, $\overline{sl}$ and $\overline{tb}$ of a positive link are large. 
On the other hand, as corollaries of MFW inequality and the Rasmussen bound on the maximal Thurston-Bennequin number \cite{Plamenevskaya, Shumakovitch}, we obtain 
\begin{align*}
\overline{sl}(L)+\overline{sl}(\overline{L})&\leq -\operatorname{breadth}_{v}P_{L}(v,z)-2, \\
\overline{tb}(L)+\overline{tb}(\overline{L})&\leq -2. 
\end{align*}
Hence, we see that the maximal self-linking number and the maximal Thurston-Bennequin number of a negative link are small. 
By this observation, it seems that Theorem~\ref{thm:self-linking} (or Corollary~\ref{cor:self-linking-homfly}) and Theorem~\ref{thm:main} (or Lemma~\ref{lem:main1}) are effective for negative links. 
In particular, in Table~\ref{table1}, we see that $-\overline{sl}(L)-1=fpbk(L)$ for negative links $L$ with up to $9$ crossings. 
\begin{question}
For any negative link $L$, does the following hold?
\[
-\overline{sl}(L)-1= fpbk(L). 
\]
\end{question}
\subsection{Non-sharpness of Theorem~\ref{thm:main}} 
For any $M>0$, there is a link $L$ such that 
\[
2fpbk(L)-(\max\{-\overline{tb}(L), -\overline{tb}(\overline{L})\}+2)>M. 
\]
In fact, we can construct such a link as follows. 
Let $K=8_{9}$. Then, $g(8_9)=3$ and $\max\{-\overline{tb}(8_9), -\overline{tb}(\overline{8_9})\}=5$. 
Let $K_n$ be the connected sum of $n$ copies of $8_9$, where we take the mirror image so that $\overline{tb}(8_9)=-5$. 
It is known that $\overline{tb}(K\sharp K')=\overline{tb}(K)+\overline{tb}(K')+1$ for any knots $K$ and $K'$ \cite{Torisu}. 
Hence, we obtain 
\begin{align*}
g(K_n)&=3n, \\
\overline{tb}(K_n)&=n\overline{tb}(8_9)+(n-1)=-5n+n-1=-4n-1. 
\end{align*}
By Theorem~\ref{thm:Hirose-Nakashima}, we see that $6n+2\leq fpbk(K_{n})$. 
Hence, we have 
\[
2fpbk(L)-(\max\{-\overline{tb}(L), -\overline{tb}(\overline{L})\}+2)\geq 2(6n+2)-(4n+1+2)=8n+1. 
\]
\subsection{Front and Lagrangian projections}
Let $F$ be a flat plumbing basket. 
If we can draw a front (or Lagrangian) projection of $\mathcal{L}_{F}$, we may use front projections of Legendrian links to study flat plumbing baskets.  
\begin{question}\label{ques:front}
Find a method to draw a front projection of the Legendrian link $\mathcal{L}_{F}$ associated with a flat plumbing basket $F$.  
\end{question}
An answer to Question~\ref{ques:front} is given in \cite{Ito-Tagami}. 
\subsection{Non-trivial open book decomposition and flat plumbing basket}
In this paper, we only consider the trivial open book decomposition $\mathcal{O}$ of $\mathbf{S}^3$. 
However, we can define flat plumbing baskets $F$ and the Legendrian knot $\mathcal{L}_{F}$ associated with $F$ in non-trivial open book decompositions of $\mathbf{S}^3$ naturally.  
\begin{question}
Consider non-trivial open book decompositions of $\mathbf{S}^3$ (or general $3$-manifold $M$), and give analogies of Theorems~\ref{thm:self-linking} and \ref{thm:main}. 

\end{question}
\subsection{Tabulation}\label{sec:tablation}
We improve \cite[Table~1]{Hirose-Nakashima} as in Table~\ref{table1}, where we use Theorems~\ref{thm:Hirose-Nakashima}, \ref{thm:self-linking} and \ref{thm:main}. 
In Table~\ref{table1}, $-\overline{tb}(K)$ means $\max\{-\overline{tb}(K), -\overline{tb}(\overline{K})\}$ and $-\overline{sl}(K)$ means $\max\{-\overline{sl}(K), -\overline{sl}(\overline{K})\}$. 
We refer to \cite[Proposition~1.6]{Ng2} for $\overline{sl}(K)$ and \cite{knot_info} for $\overline{tb}(K)$. 
Note that $fpbk(K) \in 2\mathbf{Z}$ for a knot $K$. 
\par 
In this table, the asterisks $^{\ast}$ are improved points of \cite[Table~1]{Hirose-Nakashima}.
The double asterisks $^{\ast\ast}$ are given by \cite{CCK} and Mikami Hirasawa. 
In fact, Mikami Hirasawa taught the author that a flat basket code for $8_1$ is $(1,2,4,5,3,6,1,4,6,2,5,3)$ and a flat basket code for $9_{44}$ is $(1,2,5,6,1,4,3,5,6,2,4,3)$. 
The daggers $^{\dagger}$ mean that Theorem~\ref{thm:self-linking} or ~\ref{thm:main} detects $fpbk(K)$. 
\par
For example, $8_{15}$ has a flat plumbing basket presentation $F$ with $b_{1}(F)=10$. 
On the other hand, it is known that $\max\{-\overline{sl}(8_{15}), -\overline{sl}(\overline{8_{15}})\}=11$ (see \cite[Proposition~1.6]{Ng2}). 
Hence, by Theorem~\ref{thm:self-linking}, we have $ fpbk(8_{15})= 10$. 
\par 
For another example, the knot $9_{45}$ has a flat plumbing basket presentation $F$ with $b_{1}(F)=8$. 
On the other hand, it is known that $\max\{-\overline{tb}(9_{45}), -\overline{tb}(\overline{9_{45}})\}=10$ (for example, see \cite{knot_info}). 
Moreover, $9_{45}$ is a non-torus knot. Hence, by Theorem~\ref{thm:main}, we have $10+3=13\leq 2fpbk(9_{45})\leq  16$. 
Since $fpbk(9_{45})\in 2\mathbf{Z}$, we have $fpbk(9_{45})=8$. 
\begin{question}
Determine $fpbk(K)$ for $K=9_{25}, 9_{34}, 9_{39}, 9_{40}, 9_{41}$ and $9_{43}$. 
\end{question}
%
%
%
Note that $fpbk$ is subadditive under connected sum of knots. 
However, in general, it is not additive. 
For example, Hirose-Nakashima \cite[Remark~1.4(b)]{Hirose-Nakashima} proved that $fpbk(3_1)=fpbk(\overline{3_1})=4$ but $fpbk(3_1\sharp \overline{3_1})=6$. 
Nao Kobayashi (Imoto) proved that $fpbk(6_1)=fpbk(\overline{6_1})=6$ but $fpbk(6_1\sharp \overline{6_1})=8$ in \cite[Proposition~5.4]{imoto-thesis}.
%
%

%
\begin{table}[hp]
\begin{tabular}{|c|c|c|c||c|c|c|c|}
\hline
$K$&$-\overline{tb}(K)$& $-\overline{sl}(K)$&$fpbk(K)$&$K$&$-\overline{tb}(K)$& $-\overline{sl}(K)$&$fpbk(K)$ \\ \hline\hline
$3_1$&6$^{\dagger}$&5$^{\dagger}$&4& $9_8$&8&7&8 \\ \hline
$4_1$&3&3&4&  $9_9$&16$^{\dagger}$&11$^{\dagger}$&10 \\ \hline
$5_1$&10$^{\dagger}$&7$^{\dagger}$&6&  $9_{10}$&14$^{\dagger}$&11$^{\dagger}$&$10^{\ast}$\\ \hline
$5_2$&8$^{\dagger}$&7$^{\dagger}$&6&  $9_{11}$&12$^{\dagger}$&9$^{\dagger}$&8 \\ \hline
$6_1$&5&5&6&  $9_{12}$&10$^{\dagger}$&9$^{\dagger}$&8 \\ \hline
$6_2$&7$^{\dagger}$&5&6&  $9_{13}$&14$^{\dagger}$&11$^{\dagger}$&$10^{\ast}$\\ \hline
$6_3$&4&3&6&  $9_{14}$&7&7&8 \\ \hline
$7_1$&14$^{\dagger}$&9$^{\dagger}$&8&  $9_{15}$&10$^{\dagger}$&9$^{\dagger}$&8 \\ \hline
$7_2$&10$^{\dagger}$&9$^{\dagger}$&$8^{\ast}$&  $9_{16}$&16$^{\dagger}$&11$^{\dagger}$&10 \\ \hline
$7_3$&12$^{\dagger}$&9$^{\dagger}$&8&  $9_{17}$&8&5&8 \\ \hline
$7_4$&10$^{\dagger}$&9$^{\dagger}$&$8^{\ast}$&  $9_{18}$&14$^{\dagger}$&11$^{\dagger}$&$10^{\ast}$ \\ \hline
$7_5$&12$^{\dagger}$&9$^{\dagger}$&8&  $9_{19}$&6&5&8 \\ \hline
$7_6$&8$^{\dagger}$&7$^{\dagger}$&6&  $9_{20}$&12$^{\dagger}$&9$^{\dagger}$&8 \\ \hline
$7_7$&4&5&6&  $9_{21}$&10$^{\dagger}$&9$^{\dagger}$& 8\\ \hline
$8_1$&7$^{\dagger}$&7$^{\dagger}$&$6^{\ast\ast}$&  $9_{22}$&8&5&8 \\ \hline
$8_2$&11$^{\dagger}$&7&8&  $9_{23}$&14$^{\dagger}$&11$^{\dagger}$&$10^{\ast}$\\ \hline
$8_3$&5&5&6&  $9_{24}$&6&5& 8\\ \hline
$8_4$&7&5&8&  $9_{25}$&10&9& 8--10\\ \hline
$8_5$&11$^{\dagger}$&7&8&  $9_{26}$&9&7& 8\\ \hline
$8_6$&9&7&8&  $9_{27}$&6&5& 8\\ \hline
$8_7$&8&5&8&  $9_{28}$&9&7& 8\\ \hline
$8_8$&6&5&8&  $9_{29}$&8&5& 8\\ \hline
$8_9$&5&3&8&  $9_{30}$&6&5& 8\\ \hline
$8_{10}$&8&5&8&  $9_{31}$&9&7&8 \\ \hline
$8_{11}$&9&7&8&  $9_{32}$&9&7&8 \\ \hline
$8_{12}$&5&5&6&  $9_{33}$&6&5&8 \\ \hline
$8_{13}$&6&5&8&  $9_{34}$&6&5&8--12 \\ \hline
$8_{14}$&9&7&8&  $9_{35}$&12&11$^{\dagger}$ &$10^{\ast}$\\ \hline
$8_{15}$&13&11$^{\dagger}$&$10^{\ast}$&  $9_{36}$&12$^{\dagger}$&9$^{\dagger}$& 8\\ \hline
$8_{16}$&8&5&8&  $9_{37}$&6&5& 8\\ \hline
$8_{17}$&5&3&8&  $9_{38}$&14$^{\dagger}$&11$^{\dagger}$&$10^{\ast}$\\ \hline
$8_{18}$&5&3&8&  $9_{39}$&10&9& 8--10\\ \hline
$8_{19}$&12&11$^{\dagger}$&$10^{\ast}$&  $9_{40}$&9&7&8--12 \\ \hline
$8_{20}$&6$^{\dagger}$&5&6&  $9_{41}$&7&7& 8--10\\ \hline
$8_{21}$&9$^{\dagger}$&7$^{\dagger}$&6&  $9_{42}$&5&5& 6\\ \hline
$9_1$&18$^{\dagger}$&11$^{\dagger}$&10&  $9_{43}$&10&9& 8--10\\ \hline
$9_2$&12&11$^{\dagger}$&$10^{\ast}$&  $9_{44}$&6$^{\dagger}$&5&$6^{\ast\ast}$ \\ \hline
$9_3$&16$^{\dagger}$&11$^{\dagger}$&10&  $9_{45}$&10$^{\dagger}$&9$^{\dagger}$&$8^{\ast}$\\ \hline
$9_4$&14$^{\dagger}$&11$^{\dagger}$&$10^{\ast}$&  $9_{46}$&7$^{\dagger}$&7$^{\dagger}$& 6\\ \hline
$9_5$&12&11$^{\dagger}$&$10^{\ast}$&  $9_{47}$&7&7&8 \\ \hline
$9_6$&16$^{\dagger}$&11$^{\dagger}$&10&  $9_{48}$&8$^{\dagger}$&7$^{\dagger}$& 6\\ \hline
$9_7$&14$^{\dagger}$&11$^{\dagger}$&$10^{\ast}$&  $9_{49}$&12&11$^{\dagger}$&$10^{\ast}$ \\ \hline
\end{tabular}
\caption{
Table of flat plumbing basket numbers $fpbk(K)$ for prime knots $K$ with up to $9$ crossings.  
For the notations, see Section~\ref{sec:tablation}. } 
\label{table1}
\end{table}
\par
\ 
\par
\noindent{\bf Acknowledgements: }
The author would like to thank the members of Knotting Nagoya in Nagoya Institute of Technology in June 24--25, 2017. 
In particular, the author wishes to express his gratitude to Mikami Hirasawa and Susumu Hirose for many helpful comments, data of flat plumbing baskets and their encouragements.  
The author was supported by JSPS KAKENHI Grant number 16H07230. 
%
\newpage

%
\bibliographystyle{amsplain}
\bibliography{mrabbrev,tagami}
\end{document}